\documentclass[12pt]{amsart} 
\usepackage{amssymb,amsmath} 
\usepackage{epsfig} 
\usepackage[ps2pdf,colorlinks=true,urlcolor=blue,
citecolor=red,linkcolor=blue,linktocpage,pdfpagelabels,bookmarksnumbered,bookmarksopen]{hyperref}

\newcommand{\N}{{\mathbb N}} 
 
\newcommand{\C}{{\mathbb C}} 
\newcommand{\D}{{\mathbb D}} 
\newcommand{\R}{{\mathbb R}}

\newcommand{\be}{\begin{equation}} 
\newcommand{\ee}{\end{equation}}

\newcommand{\irn}{\int_{\R^N}} 
 
\newcommand{\mS}{{\mathcal S}}
\newcommand{\mE}{{\mathcal E}}
\newcommand{\mH}{\mathcal H}

\renewcommand{\theta}{{\vartheta}} 
 
\numberwithin{equation}{section} 
\newtheorem{theorem}{Theorem}[section] 
\newtheorem{proposition}[theorem]{Proposition} 
\newtheorem{corollary}[theorem]{Corollary} 
\newtheorem{lemma}[theorem]{Lemma} 
\newtheorem{definition}[theorem]{Definition} 
\theoremstyle{definition} 
\newtheorem{remark}[theorem]{Remark} 

\newtheorem{assumption}[theorem]{Assumption}

\newcommand{\brm}{\begin{remark}\rm} 
\newcommand{\erm}{\end{remark}} 
\newcommand{\bte}{\begin{theorem}} 
\newcommand{\ete}{\end{theorem}}
\newcommand{\bas}{\begin{assumption}} 
\newcommand{\eas}{\end{assumption}} 
\newcommand{\bpr}{\begin{proposition}} 
\newcommand{\epr}{\end{proposition}}
\newcommand{\bc}{\begin{corollary}}
\newcommand{\ec}{\end{corollary}} 
\newcommand{\ble}{\begin{lemma}} 
\newcommand{\ele}{\end{lemma}} 
\newcommand{\beq}{\begin{equation}} 
\newcommand{\eeq}{\end{equation}} 
\newcommand{\bdm}{\begin{displaymath}} 
\newcommand{\edm}{\end{displaymath}} 
\numberwithin{equation}{section}

\newcommand\T{\rule{0pt}{3.1ex}}
\newcommand\B{\rule[-1.7ex]{0pt}{0pt}}

\title[On phase segregation in nonlocal two-particle Hartree 
systems]{On phase segregation in nonlocal \\ two-particle Hartree 
systems}

\author[W.H.\ Aschbacher]{Walter H.\ Aschbacher} 
\author[M.\ Squassina]{Marco Squassina} 
 
\thanks{Zentrum Mathematik, 
Technische Universit\"at M\"unchen,
Boltzmannstrasse 3, 
85747 Garching, Germany. 
E-mail: {\em aschbacher@ma.tum.de}} 
 
\thanks{Department of Computer Science, 
University of Verona, 
C\`a Vignal 2, Strada Le Grazie 15, 37134 Verona, Italy. 
E-mail: {\em marco.squassina@univr.it}}

\thanks{The research of the second author was partially supported by the 2007 MIUR national research 
project: {\it Variational and Topological Methods in the Study of 
Nonlinear Phenomena}} 
 
\subjclass[2000]{35Q40, 35J60, 35J50} 
\keywords{Coupled Hartree equations, Quantum many-body problem, Hartree approximation, 
ground states solutions, phase segregation, finite elements, self-consistent iteration}

\begin{document} 
\begin{abstract} 
We prove the phase  segregation phenomenon to occur in the ground state solutions 
of an interacting system of two self-coupled repulsive Hartree equations for 
large nonlinear and nonlocal interactions. A self-consistent numerical 
investigation visualizes the approach to this segregated regime.
\end{abstract} 
\maketitle 
 

\section{Introduction}

In this paper, we study the phase segregation phenomenon in the ground states 
of the eigenvalue system consisting of two interacting repulsive Hartree 
equations whose interaction, as well as the respective self-couplings, are 
nonlinear and nonlocal,
\begin{equation} 
\label{intro:HS} 
\begin{cases} 
-\Delta \phi_1
+V_1(x)\phi_1
+\theta_{1}\left(V\ast |\phi_1|^2\right)\phi_1
+\kappa\left(V\ast |\phi_2|^2\right)\phi_1
=\mu_1 \phi_1 \quad\text{in $\R^N$\!,}\\ 
\noalign{\vskip6pt} 
-\Delta \phi_2
+V_2(x)\phi_2
+\theta_{2}\left(V\ast |\phi_2|^2\right)\phi_2
+\kappa\left(V\ast |\phi_1|^2\right)\phi_2
=\mu_2 \phi_2 \quad\text{in $\R^N$\!,} \\ 
\noalign{\vskip6pt} 
{\|\phi_1\|}_{L^2}^2=N_1,\\
\noalign{\vskip6pt}
{\|\phi_2\|}_{L^2}^2=N_2.
\end{cases} 
\end{equation} 
Here, the external potentials $V_1$ and $V_2$  are assumed to be
nonnegative and confining, whereas the interaction potential $V$ is, 
for example, of Coulomb type. Moreover, the system is purely repulsive, 
i.e. the self-coupling constants and  the interaction strength are nonnegative,
\begin{eqnarray*}
\vartheta_1,\vartheta_2\ge 0,\,\quad
\kappa\ge 0.
\end{eqnarray*}
For fixed  $\vartheta_1,\vartheta_2\ge 0$ and fixed $N_1,N_2>0$, the phase 
segregation phenomenon in the ground state $(\phi_1,\phi_2)$ with ground state 
energy $(\mu_1,\mu_2)$ of the system \eqref{intro:HS} is 
characterized by the decay to zero of the Coulomb energy functional
\begin{eqnarray}
\label{intro:Coulomb} 
\D(\phi_1,\phi_2)
= \irn\irn  |\phi_1(x)|^2 \,V(x-y)\, |\phi_2(y)|^2\,\,
{\rm d}x\,{\rm d}y 
\end{eqnarray}
in the regime of large interaction strength $\kappa$, i.e. by
\begin{eqnarray*}
\D(\phi_1,\phi_2)=o(1)\hspace{3mm}
\mbox{for}\hspace{2mm}
\kappa\to\infty.
\end{eqnarray*}
This study is not only of independent mathematical interest but it can also be
motivated by various physical applications like, e.g.,  electromagnetic waves 
in a Kerr medium in nonlinear optics, surface gravity waves in hydrodynamics,
 and ground states in Bose-Einstein condensed bosonic quantum mechanical 
many-body systems (see also~\cite{abpr}). The latter domain has been a subject 
of great interest since many  years, both on the experimental and the 
theoretical side, starting off from a 
series of successful experimental realizations of 
Bose-Einstein condensation for atomic gases, first achieved in 1995 
for a single condensate (see e.g.~\cite{BecR}), then, in 1997, for a mixture  
of two interacting atomic species with equal masses (see e.g.~\cite{BecSR}), and, 
finally, in 2003 for triplet species states (see e.g.~\cite{triplet}).\  On the 
theoretical side, the 
standard scenario of two interacting Bose-Einstein condensates 
for a very dilute system of repulsive bosons uses the description based 
on a system of two coupled Gross-Pitaevskii equations (see e.g.\ 
\cite{esrygreene,gross,oehberg,pitae}).
These equations are formally contained in~\eqref{intro:HS} for the case of the 
zero range interaction
potential $V=\delta$, i.e.\ in the case of local nonlinearities.\footnote{
For the case of a single condensate, the stationary and dynamical 
Gross-Pitavskii equation has been rigorously derived from the many-body
bosonic Schr\"odinger equation in the weak coupling
limit, see e.g.~\cite{seiringer} and~\cite{erdos}, respectively.} 
For a complete survey paper, we also refer the reader to~\cite{Bigpap} and 
references therein.\
One may argue that, for higher density regimes, it is sensible to capture
more of the boson-boson interaction by allowing for its nonlocal and, hence,
less coarse grained resolution by use of a potential  $V\neq \delta$ (see 
e.g.~\cite{Aschbacher}). 
The phase segregation phenomenon has been studied e.g. in \cite{oehberg,riboli,
timmermans} for Gross-Pitaevskii equations, and it has been given a general 
variational
framework in~\cite{CTVpc}. Recently, the second author, jointly with M.\ Caliari,
 has investigated both 
numerically and analytically the  
behavior of ground state solutions\footnote{And, in some particular cases, also 
excited state solutions.} highlighting their location\footnote{With respect to the off-centering 
of the confining potentials $V_i$.} and the occurring phase segregation phenomena 
in the 
highly interacting regime (see~\cite{Csq1}).\footnote{For a complete numerical 
study of ground states for vector like 
nonlinear Schr\"odinger systems with cubic coupling, we also refer to
~\cite{Babao}.}

In the present paper, we extend the analysis to the nonlocal system
~\eqref{intro:HS} and 
give a proof of the phase segregation phenomenon in the 
variational calculus setup. Moreover, in contradistinction to~\cite{Csq1}, 
we adopt a classical 
self-consistent numerical approach to the solution of the ground state of
~\eqref{intro:HS} in order to compute  the 
phase segregated states  and to monitor the decay of the Coulomb 
interaction
~\eqref{intro:Coulomb}.
\vskip4pt
As we aim at keeping the paper self-contained and easily readable also for those 
readers who are more acquainted with
the physical or the numerical side, we will provide rather detailed mathematical 
arguments throughout the paper.


\section{Strong interaction and phase segregation} 
\label{SIr} 
 
Throughout this section we shall denote by $C$ a generic positive constant which can vary from line
to line inside the proofs.
 
\subsection{Functional setting}

As described in the introduction, we are interested in the case of nonnegative
confining external potentials $V_1$ and $V_2$. More precisely, we make the following assumption.

\begin{assumption}
\label{ass:ExtPot}
The external potentials $V_i$  are nonnegative, continuous, and confining, 
i.e.\ for $i=1,2$,  we have $V_i\in C(\R^N,\R_0^+)$  with
$$
\lim_{|x|\to\infty}V_i(x)=\infty.
$$
\end{assumption}\vspace{1mm}

The functional setting we want to apply makes use of the following 
Hilbert space.

\begin{definition}
Let the external potentials $V_i$ satisfy Assumption~\ref{ass:ExtPot} and let $\mH$ 
be the Hilbert subspace of $H^1(\R^N)\times H^1(\R^N)$ defined by 
\begin{eqnarray}
\label{def:H} 
\hspace{10mm}
\mH
=\Big\{(\phi_1,\phi_2)\in H^1(\R^N)\times H^1(\R^N):
\irn V_i(x)\, |\phi_i(x)|^2 \,{\rm d}x<\infty,\,\,\text{$i=1,2$}\Big\}, 
\end{eqnarray} 
where the scalar product of $\phi=(\phi_1,\phi_2)\in\mH$ with 
$\psi=(\psi_1,\psi_2)\in\mH$ is given by 
\begin{eqnarray}
\label{def:sp} 
{\langle\phi,\psi\rangle}_{\mH}
=\sum_{i=1}^2\Big(\irn \overline{\nabla\phi_i(x)}
\cdot\nabla\psi_i(x)\,{\rm d}x+\irn V_i(x)\,\overline{\phi_i(x)}\,\psi_i(x)\, {\rm d}x\Big). 
\end{eqnarray}
\end{definition}\vspace{1mm}


This functional setting is the natural framework for the study of bound states
of systems~\eqref{intro:HS} in external potentials as it allows (together with Lemma
~\ref{finitEn}) the associated
energy functional (see~\eqref{totenergf-Bis}) to be well-defined and finite.

\begin{lemma}
\label{compactnessEmb}
Under Assumption~\ref{ass:ExtPot}, for any $2\leq N\leq 5$, the embedding
$$
\mH \hookrightarrow L^{\frac{4N}{N+2}}(\R^N)\oplus L^{\frac{4N}{N+2}}(\R^N)
$$
is compact, $\mH$ being the Hilbert space~\eqref{def:H} equipped with the norm
~\eqref{def:sp}.
\end{lemma}


\begin{proof}
Let $(\phi^h_1,\phi^h_2)$ be a bounded sequence in $\mH$, say 
$\|(\phi^h_1,\phi^h_2)\|_\mH\leq C$ for all $h\in\N$. Up to a subsequence,
it converges weakly in $\mH$ to some $(\phi_1,\phi_2)\in\mH$. Moreover,
by the Rellich-Kondrachov compactness theorem, up to a further subsequence, 
$(\phi^h_1,\phi^h_2)$ converges strongly to $(\phi_1,\phi_2)$ 
in $L^2(B_R)\oplus L^2(B_R)$ for any $R>0$, where $B_R$ denotes the open ball 
in $\R^N$ of radius $R$,
centered at the origin. Let now $M>0$ be an arbitrary number. Then, by 
Assumption~\ref{ass:ExtPot}, there exists an $R>0$ such that $V_i(x)\geq M$ 
for all $x\in\R^N\setminus B_R$ and any $i=1,2$. Hence, we can write
\begin{align*}
\int_{\R^N}|\phi^h_i(x)-\phi_i(x)|^2\,{\rm d}x
&=\int_{B_R}\!|\phi^h_i(x)-\phi_i(x)|^2\,{\rm d}x
+\int_{\R^N\setminus B_R}\hspace{-2mm}|\phi^h_i(x)-\phi_i(x)|^2\,{\rm d}x \\
&\leq\int_{B_R}\!|\phi^h_i(x)-\phi_i(x)|^2\,{\rm d}x
+\frac{1}{M}\int_{\R^N\setminus B_R}\hspace{-4mm}V_i(x)\,|\phi^h_i(x)
-\phi_i(x)|^2\,{\rm d}x \\
&\leq\int_{B_R}\!|\phi^h_i(x)-\phi_i(x)|^2\,{\rm d}x
+\frac{4C^2}{M}.
\end{align*}
Let now $\varepsilon>0$ be given and choose an $M_0>0$ such that $4C^2/M_{0}<\varepsilon/2$. 
Then, as the corresponding radius $R_0>0$ is fixed, take $h_0\geq 1$ such that 
$\int_{B_{R_0}}|\phi^h_i(x)-\phi_i(x)|^2\,{\rm d}x<\varepsilon/2$
for any $h\geq h_0$. This yields $\|\phi^h_i-\phi_i\|_{L^2}\to 0$ for 
$h\to\infty$. Moreover, by the Gagliardo-Nirenberg inequality and the 
boundedness in $\mH$, we have 
$$
\|\phi^h_i-\phi_i\|^4_{L^{\frac{4N}{N+2}}}\le C\|\phi_i^h-\phi_i\|^{6-N}_{L^2}\,\|\phi_i^h-\phi_i\|^{N-2}_{H^1}\leq C\|\phi_i^h-\phi_i\|^{6-N}_{L^2},
$$ 
from which it follows that 
\begin{equation*}
\lim_{h\to\infty}{\|\phi_i^h-\phi_i\|}_{L^{\frac{4N}{N+2}}}=0.
\end{equation*}
This completes the proof. 
\end{proof}


The interaction between the components $\phi_1$ and $\phi_2$ is described by
the following Coulomb energy functional which is well-known from classical 
Hartree theory.

\begin{definition}
\label{def:DirectTerm} 
The Coulomb energy functional~\footnote{Also called {\it direct term} in the 
Hartree (-Fock) theory.}
$$
\D\!:H^1(\R^N)\times H^1(\R^N)\to\R
$$  
is defined by 
\begin{eqnarray}
\label{def:Coulomb} 
\D(\phi_1,\phi_2)
= \irn\irn |\phi_1(x)|^2\, V(x-y)\,|\phi_2(y)|^2\,\,
{\rm d}x\,{\rm d}y,
\end{eqnarray}
where the interaction potential $V$ is the Coulomb potential in $\R^N$ for 
$N\ge 3$,
\begin{eqnarray}
\label{def:V}
V(x)=\frac{1}{|x|^{N-2}}.
\end{eqnarray}
\end{definition}\vspace{1mm}

Due to the following lemma, for any $3\leq N\leq 6$, the Coulomb energy 
functional 
with potential~\eqref{def:V} is well-defined.
 
\begin{lemma} 
\label{finitEn} 
Let $3\leq N\leq 6$ and let $\phi_i\in H^1(\R^N)$ with 
$\|\phi_i\|_{L^2}^2=N_i>0$ for $i=1,2$. Then, there exists a constant 
$C$ such that 
\begin{eqnarray*}
\D(\phi_1,\phi_2)
\leq C (N_1N_2)^{\frac{6-N}{4}} 
\|\phi_1\|^{\frac{N-2}{2}}_{H^1}\,\|\phi_2\|^{\frac{N-2}{2}}_{H^1}. 
\end{eqnarray*} 
\end{lemma}

 
\begin{proof} 
Due to Schwarz' inequality, we have
\begin{eqnarray}
\label{Schwarz}
\D(\phi_1,\phi_2)^2\le \D(\phi_1,\phi_1) \,\D(\phi_2,\phi_2).
\end{eqnarray} 
Hence, by the Hardy-Littlewood-Sobolev inequality (for $N\ge 3$) and 
the Gagliardo-Nirenberg inequality (for $2\le N\le 6$), we get 
\begin{eqnarray*} 
\D(\phi_i,\phi_i)
\le C\|\phi_i\|^4_{L^{\frac{4N}{N+2}}}
\le C\|\phi_i\|^{6-N}_{L^2}\,\|\phi_i\|^{N-2}_{H^1}
=C N_i^{\frac{6-N}{2}}\,\|\phi_i\|^{N-2}_{H^1},
\end{eqnarray*} 
which yields the assertion.  
\end{proof}\vspace{1mm} 

 
\begin{remark} 
If $N\geq 2$ and the interaction potential is of the form 
\begin{eqnarray} 
\label{def:Vlambda}
V_\lambda(x)=\frac{1}{|x|^{\lambda}} \quad
\text{for some $0<\lambda<\min\{4,N\}$,}
\end{eqnarray} 
we can again estimate the energy functional 
\begin{eqnarray}
\label{def:Dlambda} 
\D_\lambda(\phi_1,\phi_2)
:= \irn\irn |\phi_1(x)|^2\, V_\lambda(x-y)\,|\phi_2(y)|^2\,{\rm d}x\,{\rm d}y
\end{eqnarray}
by virtue of the Hardy-Littlewood-Sobolev inequality (for any $0<\lambda<N$) and 
the Gagliardo-Nirenberg inequality (for any $0<\lambda\le 4$) as 
$$ 
\D_\lambda(\phi_1,\phi_2) 
\le C \|\phi_1\|^{2}_{L^{\frac{4N}{2N-\lambda}}}\|\phi_2\|^{2}_{L^{\frac{4N}{2N-\lambda}}}
\le C (N_1N_2)^{\frac{4-\lambda}{4}} \|\phi_1\|^{\lambda/2}_{H^1}\|\phi_2\|^{\lambda/2}_{H^1}.
$$ 
In particular, if $N=2$ and $\lambda=1$, we have 
$$
\D_1(\phi_1,\phi_2) \leq C (N_1N_2)^{3/4}\|\phi_1\|^{1/2}_{H^1}\|\phi_2\|^{1/2}_{H^1}.
$$ 
If $\lambda=N-2$ with $N\geq 3$ then $\D_{N-2}=\D$ and one recovers the estimate of the previous Lemma~\ref{finitEn}.
\end{remark}\vspace{3mm} 

\subsection{Existence of a minimizer} 

Let us consider the following two component Hartree eigenvalue system in 
$\R^N$ for $3\leq N\leq 6$ with Coulomb interaction $V$ from \eqref{def:V} and  
$N_1,N_2>0$,
\begin{equation} 
\label{systemHTGen} 
\begin{cases} 
-\Delta \phi_1
+V_1(x)\phi_1
+\theta_{1}\left(V\ast|\phi_1|^2\right)\phi_1
+\kappa\left(V\ast|\phi_2|^2\right)\phi_1
=\mu_1\phi_1,\\ 
\noalign{\vskip6pt} 
-\Delta \phi_2
+V_2(x)\phi_2
+\theta_{2}\left(V\ast|\phi_2|^2\right)\phi_2
+\kappa\left(V\ast|\phi_1|^2\right)\phi_2
=\mu_2\phi_2,  \\ 
\noalign{\vskip6pt} 
\|\phi_1\|^2_{L^2}=N_1,\\
\noalign{\vskip6pt}
\|\phi_2\|^2_{L^2}=N_2. 
\end{cases} 
\end{equation}

Since we are interested in the phase segregation phenomenon in the case of a 
purely repulsive Hartree system, we make
the following assumption.

\begin{assumption}
The self-coupling constants  $\theta_1, \theta_2$ and the interaction strength
$\kappa$ are nonnegative, 
$$
\theta_1,\theta_2\geq 0,\quad
\kappa\geq 0.
$$
\end{assumption}


\begin{remark}
In the case of coupled Bose-Einstein condensates discussed in the introduction 
(where the nonlinearities are local, i.e.  $V=\delta$), 
the self-coupling constants $\theta_1, \theta_2$ as well as the interaction 
strength$\kappa$ are explicitly related to the scattering lengths and the 
masses of the atomic species in the condensates (see e.g.\ \cite{esrygreene}).
\end{remark}

 
In order to study the nonlinear ground states of the Hartree system 
\eqref{systemHTGen}, we make use of the following energy functional.\footnote{From here on,  since $\theta_1,\theta_2\ge0$ and $N_1,N_2>0$ 
are fixed, we display the dependence of the energy functionals and the ground 
state energies on the interaction strength $\kappa$ only.}

\begin{definition}
The Hartree energy functional $\mE_\kappa:{\mH}\to [0,\infty)$ defined by 
\begin{equation} 
\label{totenergf-Bis} 
\mE_\kappa(\phi_1,\phi_2)
=\mE_\infty(\phi_1,\phi_2)+\kappa\,\D(\phi_1,\phi_2), 
\end{equation}
where the decoupled energy functional $\mE_\infty:\mH\to[0,\infty)$ consists of the sum of the two single particle 
energies $\mE_i:\mH\to[0,\infty)$, 
\begin{align}
\label{Einfty}
\mE_\infty(\phi_1,\phi_2)&=\sum_{i=1}^2\mE_i(\phi_i),\\ 
\label{e-ith} 
\mE_i(\phi_i)&=\irn |\nabla \phi_i(x)|^2\,{\rm d}x+\irn V_i(x)\,|\phi_i(x)|^2\, {\rm d}x+\frac{\theta_{i}}{2}\,\D(\phi_i,\phi_i). 
\end{align}
\end{definition}\vspace{2mm}

In view of Lemma~\ref{finitEn}, the functional $\mE_\kappa$ is well-defined for 
$3\leq N\leq 6$. Moreover, it is readily seen that $\mE_\kappa$ is 
a $C^1$ smooth functional and that its critical points constrained to the set 
$\{(\phi_1,\phi_2)\in {\mH}\!:\,\|\phi_i\|_{L^2}^2
=N_i\,\,\mbox{for $i=1,2$}\}$ are weak solutions of~\eqref{systemHTGen}.

\begin{remark} 
The case $\kappa=0$ corresponds to a noninteracting Hartree system
~\eqref{systemHTGen} consisting of two independent 
Hartree equations, each describing a repulsive single particle self-coupling. 
\end{remark}

\begin{definition}
\label{gse} 
The ground state energy  $E_\kappa\geq 0$  of the Hartree 
functional~\eqref{totenergf-Bis} at interaction strength $\kappa\in [0,\infty)$ is defined by 
\begin{equation} 
\label{MinKappa} 
E_\kappa=\inf_{(\phi_1,\phi_2)\in{\mathcal S}}\mE_\kappa(\phi_1,\phi_2), 
\end{equation} 
where the infimum is taken over the set 
\begin{eqnarray}
\label{def:S} 
{\mathcal S}=\{(\phi_1,\phi_2)\in {\mH}\!:\,\|\phi_i\|_{L^2}^2=N_i\,\,
\mbox{for $i=1,2$}\}. 
\end{eqnarray} 
Moreover, the segregated ground state energy $E_\infty\geq 0$ is defined by 
$$ 
E_\infty
=\inf_{(\phi_1,\phi_2)\in{\mathcal S}_\infty}\mE_\infty(\phi_1,\phi_2), 
$$ 
where now the infimum is taken over the set 
\begin{equation*} 
{\mathcal S}_\infty
=\left\{(\phi_1,\phi_2)\in {\mathcal S}\!:\, \D(\phi_1,\phi_2)=0\right\}. 
\end{equation*} 
\end{definition}\vspace{2mm} 

Let us now prove that the Hartree functional~\eqref{totenergf-Bis} admits a 
real and positive minimizer for any positive interaction strength $\kappa$. 
 
\begin{proposition}
\label{MinimizerExistence} 
Let $\kappa\in (0,\infty)$ and $3\le N\le6$. Then, there exists a positive 
minimizer
$(\phi_1^\kappa,\phi_2^\kappa)\in\mS$ of the Hartree functional
~\eqref{totenergf-Bis} 
with ground state energy $E_\kappa$ given in~\eqref{MinKappa}. 
\end{proposition}

\begin{proof}
In order to prove the assertion, we make use of the direct method in the 
calculus of variations. Hence, we verify the three standard assumptions implying 
the 
existence of a minimizer. First, since by 
Lemma~\ref{compactnessEmb},
the normed space $\mH$ from~\eqref{def:H} with the norm from \eqref{def:sp} 
is compactly embedded in $L^2(\R^N)\oplus L^2(\R^N)$, it follows that
the set $\mS$ from~\eqref{def:S} is weakly closed in $\mH$. Second, since 
\begin{eqnarray}
\label{EbN}
\|(\phi_1,\phi_2)\|_{\mH}^2
&\le& \mE_\kappa(\phi_1,\phi_2)\\
&=&\|(\phi_1,\phi_2)\|_\mH^2+\sum_{i=1}^2 \frac{\vartheta_i}{2}\,\D(\phi_i,\phi_i)
+\kappa\,\D(\phi_1,\phi_2),\nonumber
\end{eqnarray}
the set  $\{(\phi_1,\phi_2)\in\mS\!:\mE_\kappa(\phi_1,\phi_2)\le a\}$ is a 
bounded nonempty subset of $\mS$ for any positive number $a$.
Third, we have to show that the functional $\mE_\kappa$ is weakly 
lower semicontinuous 
on $\mS$. For this purpose, consider a sequence of elements 
$(\phi_1^h,\phi_2^h)\in\mS$ 
which converges for $h\to\infty$ weakly in 
${\mH}$ to some $(\phi_1,\phi_2)\in\mS$. Since, for any $i,j=1,2$, 
$$
\frac{|\phi^h_i(x)|^2|\phi^h_j(y)|^2}{|x-y|^{N-2}}\to
\frac{|\phi_i(x)|^2|\phi_j(y)|^2}{|x-y|^{N-2}}\quad
\text{for a.e.\, $(x,y)\in\R^{2N}$},
$$
Fatou's Lemma implies
\begin{equation}
\label{liminfFat}
\D(\phi_1,\phi_2)\leq\liminf_{h\to\infty}\D(\phi_1^h,\phi_2^h),\quad
\D(\phi_i,\phi_i)\leq\liminf_{h\to\infty}\D(\phi_i^h,\phi_i^h).
\end{equation}
Therefore, due to \eqref{EbN}, the fact that 
the norm ${\|\cdot\|}_\mH$ from~\eqref{def:sp} on $\mH$ is weakly lower 
semicontinuous on $\mS$, and  ~\eqref{liminfFat}, the Hartree functional 
$\mE_\kappa$ is indeed weakly lower semicontinuous on $\mS$. Hence, the three
assumptions are verified and the existence of a minimizer is proven.
Moreover, due to the convexity inequality for 
gradients,
\begin{eqnarray*}
\irn |\nabla |\phi_i|(x)|^2\,{\rm d}x
\leq \irn |\nabla \phi_i(x)|^2\,{\rm d}x, 
\end{eqnarray*}
the Hartree functional $\mE_\kappa$ satisfies the following inequality for any 
$(\phi_1,\phi_2)\in {\mathcal S}$, 
$$
\mE_\kappa(|\phi_1|,|\phi_2|)
\leq \mE_\kappa(\phi_1,\phi_2).
$$ 
Consequently, with no loss of generality, we can assume that any minimizer of 
$\mE_\kappa$ is positive.
\end{proof}

\begin{remark}
\label{weakcontin}
Note that, for $3\leq N\leq 5$, the Coulomb energy functional $\D$ 
is not only weakly lower semicontinuous as given in ~\eqref{liminfFat},  but 
even weakly continuous over $\mH$, i.e. for any sequence of elements 
$(\phi_1^h,\phi_2^h)\in\mS$ which converges for $h\to\infty$ weakly in 
${\mH}$ to some $(\phi_1,\phi_2)\in\mS$, we have
\begin{equation}
\label{wCP}
\lim_{h\to\infty}\D(\phi_1^h,\phi_2^h)=\D(\phi_1,\phi_2),\,\,\quad
\lim_{h\to\infty}\D(\phi_i^h,\phi_i^h)=\D(\phi_i,\phi_i).
\end{equation}
In order to prove this claim, we make use of Lemma~\ref{compactnessEmb}, which 
states that the embedding $\mH \hookrightarrow L^{\frac{4N}{N+2}}(\R^N)
\oplus L^{\frac{4N}{N+2}}(\R^N)$
is compact. Hence, up to a subsequence, it follows that, for $i=1,2$,
\begin{equation}
\label{LqB}
\lim_{h\to\infty}{\|\phi_i^h-\phi_i\|}_{L^{\frac{4N}{N+2}}}=0.
\end{equation}
Using \eqref{LqB}, we  want to
show that $\D(\phi_i^h,\phi_i^h)\to \D(\phi_i,\phi_i)$ as $h\to\infty$. To 
this end, we use that the Coulomb potential $V$ from \eqref{def:V} is even 
and write
\begin{equation*}
|\D(\phi_i^h,\phi_i^h)-\D(\phi_i,\phi_i)|\leq \D(||\phi_i^h|^2
-|\phi_i|^2|^{1/2},(|\phi_i^h|^2+|\phi_i|^2)^{1/2}).
\end{equation*}
By  inequality~\eqref{Schwarz}, the Hardy-Littlewood-Sobolev inequality,
and H\"older's inequality, 
it follows that there exist a constant $C$ with
\begin{align*}
|\D(\phi_i^h,\phi_i^h)-\D(\phi_i,\phi_i)|^2 
&\leq  C \|\,||\phi_i^h|^2-|\phi_i|^2|^{1/2}\|_{L^\frac{4N}{N+2}}^{4}\|
\,(|\phi_i^h|^2+|\phi_i|^2)^{1/2}\|_{L^\frac{4N}{N+2}}^{4} \\
&\leq  C \|\phi_i^h-\phi_i\|_{L^\frac{4N}{N+2}}^{2}.
\end{align*}
This implies, via~\eqref{LqB}, the desired convergence of 
$\D(\phi_i^h,\phi_i^h)$ to $\D(\phi_i,\phi_i)$.\footnote{The convergence
$\D(\phi_1^h,\phi_2^h)\to \D(\phi_1,\phi_2)$ as $h\to\infty$ can be proved in 
a similar fashion.}
Hence, it follows that all the terms in  $\mE_\kappa$ containing the Coulomb 
energy functional $\D$ are weakly continuous on $\mS$. 
Notice that, as a consequence of~\eqref{wCP}, the weak lower semicontinuity 
of $\mE_\kappa$ over $\mS$ also holds in the case of 
attractive self-coupling or attractive interaction, i.e. for 
$\vartheta_1,\vartheta_2\le 0$ or $\kappa\le0$. In fact, this case amounts to 
the replacement of some (or all) $\D$ terms (with positive coupling) in 
the Hartree functional by $-\D$.  
\end{remark}

\subsection{Phase segregation}

As pointed out in the introduction, we are interested in the situation where 
the values of the self-coupling constants $\theta_1,\theta_2$ (and $N_1,N_2$)
 remain fixed whereas the interaction strength $\kappa$ becomes very large. 
 
\begin{definition} 
\label{def:segregation} 
A sequence of minimizers $(\phi_1^\kappa,\phi_2^\kappa)\in\mS$ of the Hartree 
energy functional $\mE_\kappa$ from \eqref{totenergf-Bis}  is said to be phase 
segregating if
$$ 
\D(\phi_1^\kappa,\phi_2^\kappa)=o(1)\hspace{3mm}
\mbox{for}\hspace{2mm}
\kappa\to\infty.
$$  
\end{definition}\vspace{1mm} 

 
\begin{remark} 
If the phase segregating sequence $(\phi_1^\kappa,\phi_2^\kappa)$ is
convergent in ${\mH}$, then the limiting configuration  
$(\phi_1^\infty,\phi_2^\infty)$ satisfies $\D(\phi_1^\infty,\phi_2^\infty)=0$. 
\end{remark}\vspace{3mm} 


Let us now state our main assertion. 
 
\begin{theorem} 
\label{Gsspatseg} 
Let $3\le N\le 6$ and let $\D$ be the Coulomb energy functional from 
\eqref{def:Coulomb}.
Then, for $\kappa\in (0,\infty)$, any sequence  of minimizers $(\phi_1^\kappa,\phi_2^\kappa)\in\mS$ 
of the Hartree energy functional $\mE_\kappa$ from ~\eqref{totenergf-Bis} is 
phase segregating for 
$\kappa\to\infty$, and
\begin{eqnarray*}
\D(\phi_1^\kappa,\phi_2^\kappa)=o(\kappa^{-1}).
\end{eqnarray*}
In addition, such a sequence converges in the ${\mH}$ norm to a minimizer 
$(\phi_1^\infty,\phi_2^\infty)\in \mS_\infty$  
of the decoupled functional $\mE_\infty$ from~\eqref{Einfty} and~\eqref{e-ith}. 
\end{theorem}\vspace{3mm}

 
\begin{corollary}
\label{vie} 
Under the assumptions of Theorem \ref{Gsspatseg}, the limiting configuration 
satisfies the following set of uncoupled variational inequalities,
\begin{equation} 
\label{varineqqq} 
-\Delta \phi_i^\infty
+V_i(x)\phi_i^\infty
+\theta_{i}\left(V\ast{|\phi_i^\infty|}^2\right)\phi_i^\infty 
\leq \mu_i^\infty\phi_i^\infty,  
\end{equation} 
where  
$N_i\mu_i^\infty
=\mE_i(\phi_i^\infty)
+\frac{\theta_{i}}{2}\,\D(\phi_i^\infty,\phi_i^\infty)$ and  $i=1,2$. 
\end{corollary}\vspace{1mm}


\begin{remark} 
Although we stated Theorem~\ref{Gsspatseg} for minimizers of the Hartree 
functional in $\R^N$ with 
$3\le N\le 6$ containing the Coulomb energy functional $\D$, it also holds
for minimizers of the Hartree functional in $\R^N$ with $0<\lambda<\min\{4,N\}$ 
containing instead $\D_\lambda$ from~\eqref{def:Dlambda}.  This corresponds to 
the system
\begin{equation} 
\label{systemHTGenBis} 
\begin{cases} 
-\Delta \phi_1
+V_1(x)\phi_1
+\theta_{1}\left(V_\lambda\ast|\phi_1|^2\right)\phi_1
+\kappa\left(V_\lambda\ast|\phi_2|^2\right)\phi_1
=\mu_1\phi_1,\\ 
\noalign{\vskip6pt} 
-\Delta \phi_2
+V_2(x)\phi_2
+\theta_{2}\left(V_\lambda\ast|\phi_2|^2\right)\phi_2
+\kappa\left(V_\lambda\ast|\phi_1|^2\right)\phi_2
=\mu_2\phi_2,  \\ 
\noalign{\vskip6pt} 
{\|\phi_1\|}_{L^2}^2=N_1,\\
\noalign{\vskip6pt}
{\|\phi_2\|}_{L^2}^2=N_2,
\end{cases} 
\end{equation} 
where $V_\lambda$ stems from~\eqref{def:Vlambda}. In particular, in view of the
numerical setup, the case $N=2$ and $\lambda=1$ is covered. 
\end{remark}\vspace{1mm}


\begin{proof}
Consider a sequence of minimizers $(\phi_1^\kappa,\phi_2^\kappa)\in\mS$ 
for $\kappa\to\infty$  whose existence is assured by 
Proposition~\ref{MinimizerExistence}. 
Note first that, in the light of Definition~\ref{gse}, the sequence of 
corresponding ground state 
energies $(E_\kappa)$ is uniformly bounded because
\begin{eqnarray}
\label{gse:bound}
\hspace{10mm}E_\kappa
=\!\!\!\inf_{(\phi_1,\phi_2)\in\,\mS}\mE_\kappa(\phi_1,\phi_2)
\le \!\!\! \inf_{(\phi_1,\phi_2)\in\,\mS_\infty}\mE_\kappa(\phi_1,\phi_2)
=\!\!\!\inf_{(\phi_1,\phi_2)\in\,\mS_\infty}\mE_\infty(\phi_1,\phi_2)=E_\infty.
\end{eqnarray}
In particular, due to~\eqref{EbN} and the definition of a minimizer, the 
sequence $(\phi_1^\kappa,\phi_2^\kappa)$ is uniformly bounded in $\mH$ 
with respect to $\kappa$,
\begin{equation*}
{\|(\phi_1^\kappa,\phi_2^\kappa)\|}_{\mH}^2\le \mE_\kappa(\phi_1^\kappa,
\phi_2^\kappa) 
=E_\kappa\le E_\infty.
\end{equation*}
Hence, since $\mH$ is weakly sequentially compact, there exists a pair 
$(\phi_1^\infty,\phi_2^\infty)\in \mH$  and a subsequence 
of $(\phi_1^\kappa,\phi_2^\kappa)$, again denoted by 
$(\phi_1^\kappa,\phi_2^\kappa)$ which, for $\kappa\to\infty$,
 converges weakly in $\mH$ to $(\phi_1^\infty,\phi_2^\infty)$.
Next, we want to show that $(\phi_1^\infty,\phi_2^\infty)\in\mS_\infty$. 
Since $(\phi_1^\kappa,\phi_2^\kappa)\in\mS$ and the embedding 
$\mH\hookrightarrow L^2(\R^N)\oplus L^2(\R^N)$ is compact, we have, for $i=1,2$,
\begin{eqnarray*}
{\|\phi_i^\infty\|}_{L^2}^2=N_i.
\end{eqnarray*}
Hence, $(\phi_1^\infty,\phi_2^\infty)\in\mS$. Moreover, again due 
\eqref{EbN} and~\eqref{gse:bound}, we have
\begin{eqnarray}
\label{kappaD}
\kappa \,\D(\phi_1^\kappa,\phi_2^\kappa) 
\le \mE_\kappa(\phi_1^\kappa,\phi_2^\kappa) 
\le E_\infty,
\end{eqnarray}
which implies that the sequence $(\phi_1^\kappa,\phi_2^\kappa)$ 
is phase segregating for $\kappa\to\infty$,
\begin{equation*} 
\D(\phi_1^\kappa,\phi_2^\kappa)=O(\kappa^{-1}). 
\end{equation*}
Also, since we know that the Coulomb energy $\D$ is weakly
continuous on $\mS$, we get
\begin{equation*}
\D(\phi_1^\infty,\phi_2^\infty)=0,
\end{equation*}
and, therefore, $(\phi_1^\infty,\phi_2^\infty)\in\mS_\infty$. In order 
to prove that $(\phi_1^\infty,\phi_2^\infty)$ is a 
minimizer of $\mE_\infty$  and that $(\phi_1^\kappa,\phi_2^\kappa)$ converges 
strongly in $\mH$ to $(\phi_1^\infty,\phi_2^\infty)$, we next show that 
$\D(\phi_1^\kappa,\phi_2^\kappa)=o(\kappa^{-1})$. To this end, consider
the sequence $\kappa \,\D(\phi_1^\kappa,\phi_2^\kappa)$ which is bounded due 
to \eqref{kappaD}, and pick a convergent subsequence, 
denoted by ${\kappa_n} \,\D(\phi_1^{\kappa_n},\phi_2^{\kappa_n})$. Then, 
using that $(\phi_1^\infty,\phi_2^\infty)\in\mS_\infty$, the weak lower
semicontinuity of the decoupled energy functional $\mE_\infty$ on $\mS$, and 
\eqref{gse:bound},
we get
\begin{align}
\label{vanishing}
\hspace{13mm}
\mE_\infty(\phi_1^\infty,\phi_2^\infty)+\lim_{n\to\infty}{\kappa_n} \,
\D(\phi_1^{\kappa_n},\phi_2^{\kappa_n})
&\le\liminf_{n\to\infty}\mE_{\kappa_n}(\phi_1^{\kappa_n},\phi_2^{\kappa_n})\\
&\le E_\infty
\le \mE_\infty(\phi_1^\infty,\phi_2^\infty), \notag
\end{align}
with the consequence that ${\kappa_n} \,\D(\phi_1^{\kappa_n},\phi_2^{\kappa_n})
=0$ as $n\to\infty$. 
Therefore, since this holds for all convergent subsequences
of $\kappa \,\D(\phi_1^\kappa,\phi_2^\kappa)$, we arrive at
\begin{eqnarray}
\label{okappa}
\D(\phi_1^\kappa,\phi_2^\kappa)=o(\kappa^{-1}).
\end{eqnarray}
This implies, on one hand, that  $(\phi_1^\kappa,\phi_2^\kappa)$ converges 
strongly in 
$\mH$ to $(\phi_1^\infty,\phi_2^\infty)$ since from 
$E_k\le \mE_\infty(\phi_1^\infty,\phi_2^\infty)$, \eqref{okappa}, 
and the weak continuity of $\D(\phi_i^\kappa,\phi_i^\kappa)$, we get
\begin{eqnarray*}
\limsup_{\kappa\to\infty} {\|(\phi_1^\kappa,\phi_2^\kappa)\|}_{\mH}^2
\le {\|(\phi_1^\infty,\phi_2^\infty)\|}_{\mH}^2.
\end{eqnarray*}
On the other hand, using~\eqref{vanishing} and~\eqref{okappa}, we see that 
$(\phi_1^\infty,\phi_2^\infty)$ is a minimizer of $\mE_\infty$, that is $E_\infty=\mE_\infty(\phi_1^\infty,\phi_2^\infty)$.
Finally, we note that, again due to~\eqref{vanishing}, we have $E_\kappa\to E_\infty$ as $\kappa\to\infty$. 
This brings the proof of Theorem~\ref{Gsspatseg} to an end.
\end{proof}\vspace{2mm}


Finally, we prove the assertion of Corollary~\ref{vie}.
\begin{proof}
Observe that, by virtue of  
\begin{equation} 
\label{eig-systemHTGen}
\mu_i^\kappa
=\frac{1}{N_i}\Big\{\mE_i(\phi_i^\kappa)+\frac{\theta_{i}}{2}\, 
\D(\phi_i^\kappa,\phi_i^\kappa) +\kappa\, \D(\phi_1^\kappa,\phi_2^\kappa)
\Big\}, 
\end{equation} 
the inequality \eqref{kappaD} and $\mE_i(\phi_i^\kappa)\leq E_\infty$, we get 
\begin{equation*} 
\sup_{\kappa\geq 1}\mu_i^\kappa<\infty, 
\end{equation*} 
where $\mu_i^\kappa$ denotes the nonlinear eigenvalue of the minimizer
$\phi_i^\kappa$ as the weak nonlinear ground state in the corresponding
nonlinear eigenvalue system~\eqref{systemHTGen}. Then, up to a subsequence, $\mu_i^\kappa\to \mu_i^\infty$ as $\kappa\to\infty$. Testing 
the equations of~\eqref{systemHTGen} with arbitrary nonnegative functions $\eta$ of compact support, we get, recalling that $\phi_i\geq 0$,
\begin{align*} 
\irn\nabla \phi_i^\kappa(x)\cdot\nabla\eta(x) \,{\rm d}x 
&+\irn V_i(x)\,\phi_i^\kappa(x)\,\eta(x)\,{\rm d}x\\ 
&+\theta_{i}\irn\irn \frac{\phi_i^\kappa(y)^2\,\phi_i^\kappa(x)\,\eta(x) }{|x-y|^{N-2}}\,\,{\rm d}x\,{\rm d}y\\ 
&\leq \mu_i^\kappa\irn \phi_i^\kappa(x)\,\eta(x)\,{\rm d}x. 
\end{align*} 
Hence, letting $\kappa\to\infty$, it turns out that $\phi_i^\infty$ 
satisfies the variational inequality~\eqref{varineqqq}. 
Finally, the strong convergence and~\eqref{eig-systemHTGen} yields, 
for $i=1,2$, 
$$
N_i\,\mu_i^\infty
=\mE_i(\phi_i^\infty)
+\frac{\theta_{i}}{2}\,\D(\phi_i^\infty,\phi_i^\infty).
$$ 
This ends the proof of Corollary~\ref{vie}. 
\end{proof} 

 
\medskip 
\section{Numerical approach} 

 
\subsection{Galerkin approximation of the nonlinear eigenvalue 
system} 
 
In order to carry out the numerical simulation, we treat the  
Hartree system in the plane from \eqref{systemHTGenBis} with $N=2$  
and $\lambda=1$ in the framework of the following finite element  
approximation.   
As physical subdomain of $\R^2$, we choose the open square  
\begin{eqnarray}
\label{def:Omega}
\Omega=(0,D)^{\times 2}
\end{eqnarray}
 with $D>0$ whose closure is the union  
of the $(m-1)^2$ congruent closed subsquares generated by dividing  
each side of $\Omega$ equidistantly into $m-1$ intervals. Let us  
denote by $M=(m-2)^2$ the total number of interior vertices of  
this  
lattice and by $h=D/(m-1)$ the lattice spacing.\footnote{\label{foot:bijection}As  
bijection from the one-dimensional to the two-dimensional  
lattice numbering, we may use the mapping  
$\tau:\{0,...,m-1\}^{\times 2} 
\to \{0,...,m^2-1\}$ with $j=\tau(m_1,m_2):=m_1+m_2m$.}  
Moreover, let us   
choose the  Galerkin space $S_h$ to be spanned by the bilinear  
Lagrange rectangle finite elements  
$\varphi_j\in C(\overline{\Omega})$.\footnote{The reference  
basis function $\varphi_0:\overline{\Omega}\to [0,\infty)$  
is defined on  
its support $[0,2h]^{\times 2}$ by 
\begin{eqnarray} 
\label{def:blfe} 
\varphi_0(x,y) 
:=\frac{1}{h^2}\left\{ 
\begin{array}{rl} 
xy,           & \mbox{if } (x,y)\in [0,h]^{\times 2},\\ 
(2h-x)y,      & \mbox{if } (x,y)\in [h,2h]\times [0,h],\\ 
(2h-x)(2h-y), & \mbox{if } (x,y)\in [h,2h]^{\times 2},\\ 
x(2h-y),      & \mbox{if } (x,y)\in [0,h]\times [h,2h], 
\end{array} 
\right. 
\end{eqnarray} 
see Figure \ref{fig:refel}. The functions $\varphi_j$ are then defined to 
be of the form \eqref{def:blfe} having their support translated by 
$(m_1 h,m_2 h)$ with $m_1,m_2=0,...,m-3$.  
} 
\begin{figure}[h!] 
\begin{center} 
\includegraphics[width=4cm,height=4.5cm]{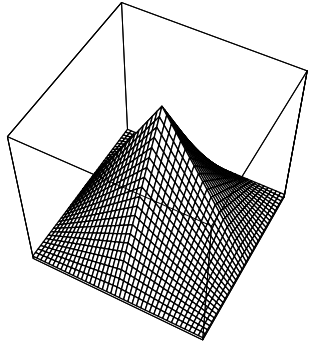} 
\end{center} 
\caption{$\varphi_0(x,y)$ on its support $[0,2h]^{\times 2}$ with  
maximum at vertex $(h,h)$.} 
\label{fig:refel} 
\end{figure}
 
Hence, with this choice, we have 
$$
S_h\subset C(\overline{\Omega})\cap  
H_0^1(\Omega)
$$ 
and $\dim S_h=(m-2)^2$.
The Hartree system  
\eqref{systemHTGenBis} in its weak finite element approximation form  
reads, for all $\varphi\in S_h$,\footnote{In this section, $(\cdot,\cdot)$ and $\|\cdot\|$ stand for  the $L^ 2(\Omega)$-scalar product and $L^ 2(\Omega)$-norm,  
respectively. Moreover, the convolution on the  
finite domain $\Omega$ is defined, for any $(x,y)\in\Omega$, by
$(V\ast \phi)(x,y):= 
\int_\Omega V(x-x',y-y')\,\phi(x',y')  \,{\rm d}x'\,{\rm d}y'$.}
\begin{equation} 
\label{FEsystem} 
\begin{cases} 
(\nabla\varphi,\nabla\phi_1) 
+(\varphi,V_1\phi_1) 
+(\varphi,(V\ast [\vartheta_1|\phi_1|^2+\kappa|\phi_2|^2])\phi_1) 
=\mu_{1,0}  \left(\varphi,\phi_1\right),\\ 
(\nabla\varphi,\nabla\phi_2) 
+(\varphi,V_2\phi_2) 
+(\varphi,(V\ast[\vartheta_2|\phi_2|^2+\kappa|\phi_1|^2])\phi_2) 
=\mu_{2,0} (\varphi,\phi_2),\\ 
\|\phi_1\|^2_{L^2}=N_1,\\ 
\|\phi_2\|^2_{L^2}=N_2. 
\end{cases} 
\end{equation} 
If we expand   
$\phi_\alpha\in S_h$  with expansion coefficients $z_\alpha= 
(z_{\alpha,1},...,z_{\alpha,M})\in\C^M$ 
w.r.t. the finite element basis $\{\varphi_j\}_{j=1}^M$ of $S_h$, 
\begin{eqnarray} 
\label{expansion} 
\phi_\alpha=\sum_{j=1}^{M}z_{\alpha,j}\,\varphi_j, 
\end{eqnarray} 
and if we plug this expansion together with $\varphi=\varphi_i$, 
$i=1,...,M$, into \eqref{FEsystem}, we get the following  
coupled matrix system on $\C^M$,\footnote{ 
$|z|_2^2:=\langle z,z\rangle_2$ denotes the $L^2$-norm 
on $\C^M$, where $\langle z,w\rangle:=\sum_{j=1}^M 
\overline{z}_jw_j$ and $\langle z,w\rangle_2:= 
\langle z,Aw\rangle$ are the Euclidean and the  $L^2$-scalar  
product on $\C^M$, respectively.  For $\phi_\alpha$ from  
\eqref{expansion}, we have $\|\phi_\alpha\|^2= 
\langle z_\alpha,Az_\alpha\rangle=|z_\alpha|_2^2=N_\alpha$  
for  
$\alpha=1,2$.} 
\begin{equation} 
\begin{cases} 
\label{Qsystem} 
Q_1[z_1,z_2]z_1=\mu_{0,1}z_1,\\ 
Q_2[z_1,z_2]z_2=\mu_{0,2}z_2,\\ 
|z_1|_2^2=N_1,\\ 
|z_2|_2^2=N_2. 
\end{cases} 
\end{equation} 
Here, the matrix-valued mappings $Q_1,Q_2: \C^M\times\C^M\to\C^{M\times M}$  
 are defined by 
\begin{eqnarray*} 
Q_1[z_1,z_2]&:=&A^{-1}(B+Y_1+\vartheta_1\, G[z_1]+\kappa\,  G[z_2]),\\ 
Q_2[z_1,z_2]&:=&A^{-1}(B+Y_2+\vartheta_2\, G[z_2]+\kappa\,  G[z_1]), 
\end{eqnarray*} 
and $A\in\C^{M\times M}$ is the mass matrix, $B\in\C^{M\times M}$ the  
stiffness matrix, and $Y_\alpha\in\C^{M\times M}$ the matrices generated by  
the external potentials $V_\alpha$,  
\begin{eqnarray} 
\label{def:A,B,Y} 
A_{ij}:=(\varphi_i,\varphi_j),\qquad 
B_{ij}:=(\nabla\varphi_i,\nabla\varphi_j),\qquad 
(Y_\alpha)_{ij}:=(\varphi_i,V_\alpha\varphi_j). 
\end{eqnarray} 
Moreover, the matrix-valued mapping $G:\C^{M}\to \C^{M\times M}$ is defined  
on $w=(w_1,...,w_{M})\in\C^{M}$ by 
\begin{eqnarray*} 
G[w]_{ij}:=\left(\varphi_i,g\big[\sum_{k=1}^{M}w_k\,\varphi_k\big]\varphi_j\right) 
= \sum_{k,l=1}^{M}\overline{w}_k w_l \,V_{iklj}, 
\end{eqnarray*} 
where the function $g$ and the Hartree  
convolution term $V_{iklj}$ are defined by  
\begin{eqnarray} 
\label{def:g,V} 
g[\phi]:=V\ast|\phi|^2,\qquad 
V_{iklj}:=(\varphi_i,V\ast(\overline{\varphi}_k\varphi_l)\varphi_j). 
\end{eqnarray}

\brm 
\label{rem:MassLumping} 
We avoid the inversion of the mass matrix $A$ and simplify 
the evaluation of the double integral in the Hartree convolution term  
\eqref{def:g,V} by approximating the integrals over $\Omega$ by the  
standard mass lumping quadrature procedure.  
\erm\vspace{1mm} 

 
In order to simplify the eigenvalue system \eqref{Qsystem}  
with the help of Remark \ref{rem:MassLumping},  
let  us introduce the mappings $H_1,H_2:  
\C^M\times\C^M\to\C^{M\times M}$ defined by 
\begin{eqnarray*} 
H_1[z_1,z_2]&:=&  
\frac{1}{h^2}(B+Y_1+\vartheta_1\,{\rm diag}(G_0[z_1]) 
+\kappa\,{\rm diag}(G_0[z_2])),\\ 
H_2[z_1,z_2]&:=&  
\frac{1}{h^2}(B+Y_2 
+\vartheta_2\,{\rm diag}(G_0[z_2]) 
+\kappa\,{\rm diag}(G_0[z_1])), 
\end{eqnarray*} 
where ${\rm diag}:\C^M\to\C^{M\times M}$ is defined to be the matrix-valued  
mapping on $w=(w_1,...,w_M)\in\C^M$ defined by  ${\rm diag}(w)_{ij}:= 
\delta_{ij}w_j$ for all $i,j=1,...,M$, and $G_0:\C^M\to\C^M$  
is defined  by 
\begin{eqnarray*} 
G_0[w]_i:=h^4\sum_{j=1}^M\,|w_j|^2\,V(h(\tau^{-1}(i)-\tau^{-1}(j))), 
\end{eqnarray*} 
where $\tau$ is the grid numbering bijection from footnote 
\ref{foot:bijection}. Hence,  
Remark \ref{rem:MassLumping} amounts to the replacement  
$G[z]\mapsto {\rm diag}(G_0[z])$ and we get approximated 
Hartree system 
\begin{equation} 
\begin{cases} 
\label{Hsystem} 
H_1[z_1,z_2]z_1=\mu_{0,1}z_1,\\ 
H_2[z_1,z_2]z_2=\mu_{0,2}z_2,\\ 
|z_1|_2^2=N_1,\\ 
|z_2|_2^2=N_2. 
\end{cases} 
\end{equation}\vspace{1mm} 
 

\subsection{Algorithms} 
 
In order to solve the nonlinear coupled eigenvalue system  
\eqref{Hsystem}, we make use of the method of  
successive substitution\footnote{Also called nonlinear Richardson  
iteration or Picard iteration.} whose  fixed-point map  
is constructed with the 
help of the power method used for the solution of the corresponding  
linearized problem. In the following, we briefly describe the basic 
ideas of these algorithms.\\

 
{\sc Method of successive substitution (MSS)} \\ 
Let $\mathcal{M}$ be the compact set  
$$ 
\mathcal{M}:=\{[z_1,z_2]\in \C^M\times\C^M\,|\, 
|z_1|_2^2=N_1,\,|z_2|_2^2=N_2\}. 
$$ 
The MSS is an iterative method of the form  
\begin{eqnarray} 
\label{def:F} 
[z_1^{(n+1)},z_2^{(n+1)}]=F[z_1^{(n)},z_2^{(n)}], 
\end{eqnarray} 
where the fixed point map 
$F:\mathcal{M}\to \mathcal{M}$ is constructed as follows. 
Given an approximate nonlinear  
system ground state  
$[z_1^{(n)},z_2^{(n)}]\in \mathcal{M}$  at iteration level  
$n\in\N$, the approximate nonlinear system ground state   
$[z_1^{(n+1)},z_2^{(n+1)}]\in \mathcal{M}$ at iteration level  
$n+1$ is defined to be the linear system ground state of the  
linearized eigenvalue system 
\begin{equation} 
\begin{cases} 
\label{Hlinsystem} 
H_1[z_1^{(n)},z_2^{(n)}]z_1^{(n+1)} 
=\epsilon_{0,1}^{(n+1)} z_1^{(n+1)},\\  
H_2[z_1^{(n)},z_2^{(n)}]z_2^{(n+1)} 
=\epsilon_{0,2}^{(n+1)} z_2^{(n+1)},\\ 
|z_1^{(n+1)}|_2^2=N_1,\\ 
|z_2^{(n+1)}|_2^2=N_2.\\ 
\end{cases} 
\end{equation} 


\brm 
Here and in the following, we make the assumption that  
$H_\alpha[z_1^{(n)},z_2^{(n)}]$ has a unique linear ground state. 
E.g., using perturbation theory in the regime of small nonlinear  
couplings, this holds as soon as the linear operator 
$H_\alpha[0,0]$ has a nondegenerate ground state energy.  
\erm 


\brm 
We can write the fixed point map $F_\alpha[z_1^{(n)},z_2^{(n)}]$  
from \eqref{def:F} with the help of the linear ground state  
projection 
$$ 
P_\alpha[z_1,z_2] 
=-\frac{1}{2\pi {\rm i}}\oint_{\Gamma_\alpha[z_1,z_2]} 
(H_\alpha[z_1,z_2]-\zeta)^{-1}\,\, {\rm d}\zeta, 
$$ 
where $\Gamma_\alpha[z_1,z_2]$ is a path which encircles the  
linear ground state energy of $H_\alpha[z_1,z_2]$ in  
the positive  
direction 
and no other point of the spectrum of $H_\alpha[z_1,z_2]$.  
The map  
$F$ can now be written as 
\begin{eqnarray} 
\label{F:projector} 
F_\alpha[z_1^{(n)},z_2^{(n)}] 
=\sqrt{N_\alpha}\,\, 
\frac{P_\alpha[z_1^{(n)},z_2^{(n)}]z_\alpha^{(n)}} 
{|P_\alpha[z_1^{(n)},z_2^{(n)}]z_\alpha^{(n)}|_2}\,. 
\end{eqnarray} 
Since $H_\alpha[\cdot,\cdot]$ is Lipschitz continuous on $\mathcal{M}$,  
the map $F$ 
has a not necessarily unique fixed point due to Schauder's fixed point  
theorem. 
\erm 
 

 

The system \eqref{Hlinsystem} being 
not only linearized but also decoupled, we can solve  
the two linear eigenvalue problems separately. In order to  
approximately determine the ground states of the linear eigenvalue  
problems, we make use of the power method which works as  
follows.\\ 
 
 
{\sc Power method (PM)} \\ 
The PM computes the eigenvector of  
$H_\alpha[z_1^{(n)},z_2^{(n)}]$ whose eigenvalue has largest  
modulus amongst all the eigenvalues whose eigenvectors 
 appear in the eigenvector expansion of the starting approximation. To access the ground state of $H_\alpha[z_1^{(n)},z_2^{(n)}]$, 
we apply the following spectral shift\footnote{ 
For $A=[a_{ij}]\in\C^{M\times M}$, we define the $\ell^1$- 
matrix norm by ${|A|}_1:=\sum_{i,j=1}^M|a_{ij}|$.} 
$$
s_\alpha^{(n)}:={|H_\alpha[z_1^{(n)},z_2^{(n)}]|}_1+1.
$$
 Moreover, we  
define the  
shifted operator by  
$$ 
\widehat{H}_\alpha[z_1^{(n)},z_2^{(n)}] 
:=H_\alpha[z_1^{(n)},z_2^{(n)}]-s_\alpha^{(n)}. 
$$ 
Now,  the $p$-th iterate of  
the PM  iteration is defined by\footnote{$|z|:=\langle z,z\rangle^{1/2}$  
denotes the Euclidean norm of $z\in\C^M$.} 
\begin{eqnarray} 
\label{def:PM} 
z_\alpha^{(n+1),p} 
&:=&\frac{{\widehat{H}_\alpha[z_1^{(n)},z_2^{(n)}]}^p z_\alpha^{(n)}} 
{{|\widehat{H}_\alpha[z_1^{(n)},z_2^{(n)}]}^p z_\alpha^{(n)}|}. 
\end{eqnarray} 

\brm 
Since $\widehat{H}_\alpha[z_1^{(n)},z_2^{(n)}]$ is real symmetric,  
the spectral theorem implies the existence  
of an orthonormal basis of $\C^M$ of eigenvectors  
$\{w_{\alpha,k}\}_{k=0}^{M-1}$ of  
$\widehat{H}_\alpha[z_1^{(n)},z_2^{(n)}]$.\footnote{We suppress the  
superscript 
$n$ in the eigenvectors, eigenvalues, and in the expansion coefficients,  
since  
the PM iteration acts at a fixed $n$. Moreover, the numbering  
starts at  
$0$ being the index of the ground state.}  Moreover, since the spectral  
radius  
of $\widehat{H}_\alpha[z_1^{(n)},z_2^{(n)}]$ is smaller than  
$s_\alpha^{(n)}$, we have for all eigenvalues of the shifted operator  
$\hat{\epsilon}_{\alpha,k} 
\in{\rm spec}(H_\alpha[z_1^{(n)},z_2^{(n)}])-s_\alpha^{(n)}$ that 
$$ 
-2s_\alpha^{(n)} 
< \hat{\epsilon}_{\alpha,0} 
< \hat{\epsilon}_{\alpha,1} 
\le ... 
\le \hat{\epsilon}_{\alpha,M-1} 
<0.
$$ 
Let us expand $z_\alpha^{(n)}$ w.r.t. the orthonormal  basis   
$\{w_{\alpha,k}\}_{k=0}^{M-1}$ as  
$$ 
z_\alpha^{(n)}=\sum_{k=0}^{M-1} 
\xi_{\alpha,k}w_{\alpha,k}, 
$$  
use that $\widehat{H}_\alpha[z_1^{(n)},z_2^{(n)}] w_{\alpha,k} 
=\hat{\epsilon}_{\alpha,k}\,w_{\alpha,k}$, and divide the  
numerator and the denominator in \eqref{def:PM} 
by $|\hat{\epsilon}_{\alpha,0}|^p$. Moreover, let us assume 
that  $\xi_{\alpha,0}\neq 0$. Then, in the large $p$ limit, 
$(-1)^p z_\alpha^{(n+1),p}$ converges to a multiple of 
the ground state 
of $H_\alpha[z_1^{(n)},z_2^{(n)}]$,  
\begin{eqnarray} 
\label{PM:convergence} 
z_\alpha^{(n+1),p} 
=(-1)^p\frac{\xi_{0,\alpha}}{{\big|\xi_{0,\alpha}\big|}}\, 
w_{0,\alpha}+o(1). 
\end{eqnarray} 
\erm

\brm 
Using formula \eqref{PM:convergence}, the fixed point map $F$ from 
\eqref{def:F},  
\eqref{F:projector} can also be written as 
\begin{eqnarray*} 
F_\alpha[z_1^{(n)},z_2^{(n)}] 
=\sqrt{N_\alpha}\,\lim_{p\to\infty} 
\frac{(-1)^p z_\alpha^{(n+1),p}}{{|z_\alpha^{(n+1),p}|}_2}. 
\end{eqnarray*} 
\erm 

 
{\sc Stopping criteria}\\ 
Both for the inner PM iteration and the outer MSS iteration, we use a  
relative error stopping criterion in  
the numerical computation. For the PM iteration, let us define the energy 
\begin{eqnarray} 
\label{def:LinEnergy} 
\hat{\epsilon}_{\alpha,0}^{(n+1),p} 
:=\langle z_\alpha^{(n+1),p}, 
\widehat{H}_\alpha[z_1^{(n)},z_2^{(n)}]z_\alpha^{(n+1),p}\rangle. 
\end{eqnarray} 
Then, for suitably chosen accuracy tolerance  
$\delta_{\rm PM}>0$, we stop the PM iteration  
for each component $\alpha=1,2$ as soon as 
\begin{eqnarray} 
\label{StopPM} 
\frac{|(\widehat{H}_\alpha[z_1^{(n)},z_2^{(n)}]- 
\hat{\epsilon}_{\alpha,0}^{(n+1),p})\, 
z_\alpha^{(n+1),p}|} 
{|\hat{\epsilon}_{\alpha,0}^{(n+1),p}+s_\alpha^{(n)}|} 
\le \delta_{\rm PM}. 
\end{eqnarray} 

 
\brm 
Note that the quotient \eqref{StopPM} does not depend on the shift  
$s_\alpha^{(n)}$, since the iterates $z_\alpha^{(n+1),p}$ are normalized  
w.r.t. the Euclidean norm on $\C^M$. 
\erm
 
\brm 
Clearly, the stopping criterion \eqref{StopPM} is satisfied for any  
eigenvector of $\widehat{H}_\alpha[z_1^{(n)},z_2^{(n)}]$. But as soon  
as $\xi_{\alpha,0}\neq 0$, e.g. due to finite precision arithmetic,  
the PM iterate $z_\alpha^{(n+1),p}$ converges to a multiple 
of the ground state $w_{\alpha,0}$. But note that the chosen accuracy 
may be reached before a nonvanishing $\xi_{\alpha,0}$ is generated. 
\erm 


For the MSS iteration, we implement a similar stopping criterion. To this end, 
 we define the approximate nonlinear ground state energies  
as 
\begin{eqnarray*} 
\mu_{\alpha,0}^{(n+1)} 
&:=&\frac{1}{N_\alpha}\, 
\langle z_\alpha^{(n+1)},H_\alpha[z_1^{(n+1)},z_2^{(n+1)}] 
z_\alpha^{(n+1)}\rangle_2, 
\end{eqnarray*} 
where, compared to \eqref{def:LinEnergy}, the Hartree energy $H_\alpha$ 
depends on  iteration level $n+1$ instead of level $n$. We stop the  
MSS iteration as soon as 
\begin{eqnarray*} 
\frac{{|(H_\alpha[z_1^{(n+1)},z_2^{(n+1)}]- \mu_{0,\alpha}^{(n+1)})\, 
z_\alpha^{(n+1)}|}_2} 
{|\mu_{0,\alpha}^{(n+1)}|} 
\le \delta_{\rm MSS}, 
\end{eqnarray*} 
where $\delta_{\rm MSS}>0$ is some suitably chosen accuracy tolerance.

\subsection{Phase segregation} 
 
As it has been defined above in Definition \ref{def:segregation},  
a sequence of  
nonlinear ground state solutions  
$(\phi_1^\kappa, \phi_2^\kappa)$ is phase segregating if its Coulomb energy
vanishes in the limit of large interaction strength $\kappa$, i.e.
\begin{eqnarray} 
\label{def:D} 
\D(\phi_1^\kappa, \phi_2^\kappa) 
=(\phi_1^\kappa,(V\ast|\phi_2^\kappa|^2)\phi_1^\kappa)\to 0 \hspace{3mm}
\mbox{for}\hspace{2mm}
\kappa\to\infty.
\end{eqnarray} 
Plugging the expansions \eqref{expansion} into \eqref{def:D}, we get 
$$ 
 \D\Big(\sum_{i=1}^Mz_{1,i}\,\varphi_i, \sum_{j=1}^M 
z_{2,j}\,\varphi_j\Big)=\langle z_1, G[z_2] z_1\rangle.  
$$ 
Hence, making use  
of Remark \ref{rem:MassLumping}, we define the approximated 
Coulomb energy  $\D_0:\C^M\times\C^M\to\R$ by 
\begin{eqnarray} 
\label{def:D0} 
\D_0[z_1,z_2]:=\langle z_1,{\rm diag}(G_0[z_2])z_1\rangle. 
\end{eqnarray} 
Below, we will use this approximation  in the numerical 
computation of the Coulomb energy.

\bigskip
\section{Figures} 

The numerical computations leading to the following figures visualize the
qualitative picture of the approach to the 
segregated regime. First, 
we exhibit the densities of the wave functions   
$\phi_1$ and $\phi_2$ for increasing values of the interaction strength 
$\kappa$ approaching the segregated regime. Second, we report on the  
decay of the Coulomb energy \eqref{def:D0}.\\
 
We choose the external potentials $V_\alpha$ for $\alpha=1,2$  to be isotropic 
harmonic 
 potentials,
\begin{eqnarray}
\label{def:Valpha}
V_\alpha(x,y)
=c_\alpha \left((x-a_\alpha)^2+(y-b_\alpha)^2\right),
\end{eqnarray}
and the interaction potential $V$ to be a regularized Yukawa potential,
\begin{eqnarray}
\label{def:Vnum}
V(x,y)
=\frac{{\rm e}^{-\Gamma\sqrt{x^2+y^2}}}{\sqrt{x^2+y^2}+\gamma}\,.
\end{eqnarray}

\brm 
The potential \eqref{def:Vnum} being the regularized three-dimensional Yukawa
potential, it may be argued that we consider a physically three-dimensional
system  constrained to a two-dimensional submanifold of the
three-dimensional configuration space.
\erm\vspace{1mm}

The specification of the parameters used in the simulations below
is summarized in the following table (cf. \eqref{def:Omega}, 
\eqref{FEsystem}, \eqref{def:Valpha}, and \eqref{def:Vnum}).\footnote{The code
is part of our {\sf Hartree} package written in C++.}
\vspace{3mm}

\begin{center}
\begin{tabular}{|l|l|l|l|l|l|l|l|l|l|l|l|l|}\hline
$N_1$ \T \B & 
$N_2$       & 
$a_1$       & 
$b_1$       & 
$c_1$       & 
$a_2$       & 
$b_2$       & 
$c_2$       &
$\theta_1$  &
$\theta_2$  &
$\kappa$    &
$\Gamma$    &
$\gamma$  
\\\hline
$1$ \T \B    & $1$ &
$D/2$ & $D/2$ & $10^5$ & 
$D/2$ & $D/2$ & $10^3$ & 
$0$ & $0$ & cf. below & $10^2$ & $10^{-1}$ 
\\\hline
\end{tabular}  
\end{center} 

\vspace{2mm}
\brm  
All the qualitative features of the following simulations have been tested 
for  stability in different physical and numerical parameter ranges.
\erm   

\subsection{$\kappa=0$}

For the interaction strength $\kappa=0$, the system is uncoupled and linear, 
and we find the ground state wave functions of the harmonic oscillator. The 
supports are fully overlapping, see Figure \ref{fig:seg0}.\footnote{All the 
figures have been  produced  with the help of {\sf gnuplot}.}  

\begin{figure}[h!]
\centering  
\includegraphics[width=6.5cm,height=4.5cm]{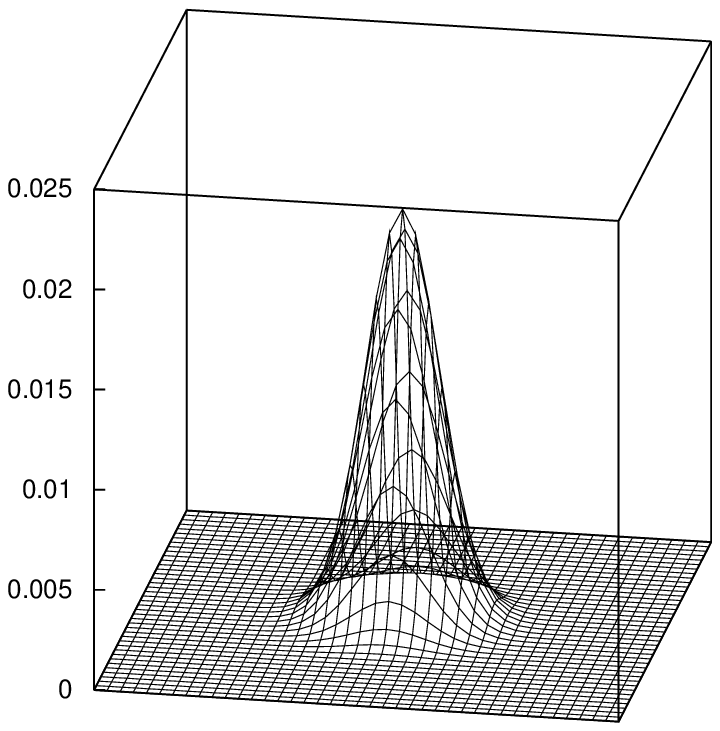}
\hspace{-1.5cm}  
\includegraphics[width=6.5cm,height=4.5cm]{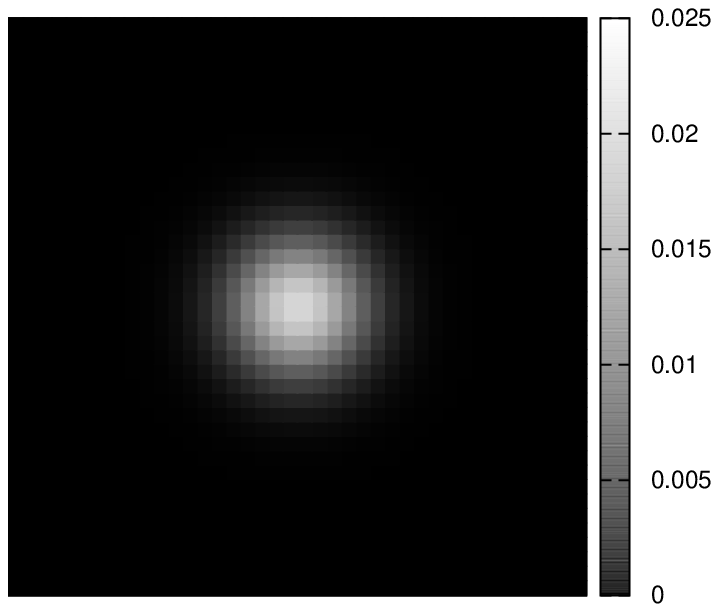}

\hspace{1cm}  
  
\vspace{-1.2cm} 

\includegraphics[width=6.5cm,height=4.5cm]{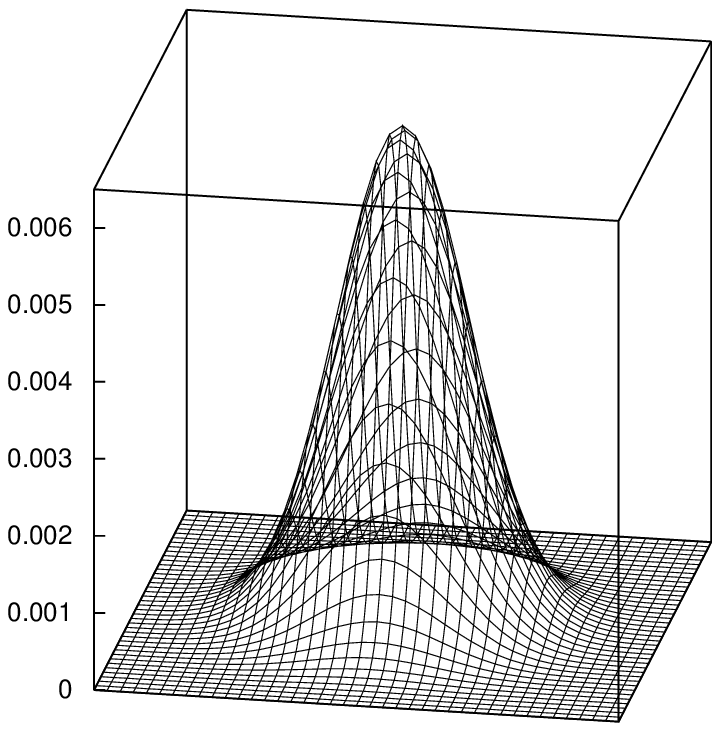}
\hspace{-1.5cm}  
\includegraphics[width=6.5cm,height=4.5cm]{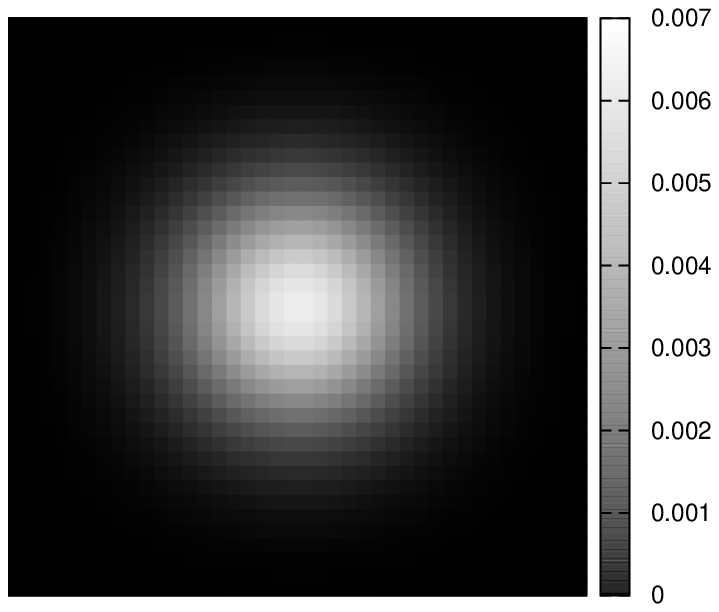}
\caption{The wave function densities $|z_\alpha|^2$ and their contours with  
$\alpha=1$ above and $\alpha=2$ below for the interaction strength $\kappa=0$.}  
\label{fig:seg0}  
\end{figure} 

\subsection{$\kappa=0.5$} 

The wave functions $\phi_1$ and $\phi_2$ start  
to feel their respective repulsion. The support of $\phi_1$ is retracting  
whereas the one of $\phi_2$ gets pushed outwards. The supports are still  
heavily overlapping, see Figure \ref{fig:seg0.5}.
\begin{figure}[h!]
\centering  
\includegraphics[width=6.5cm,height=4.5cm]{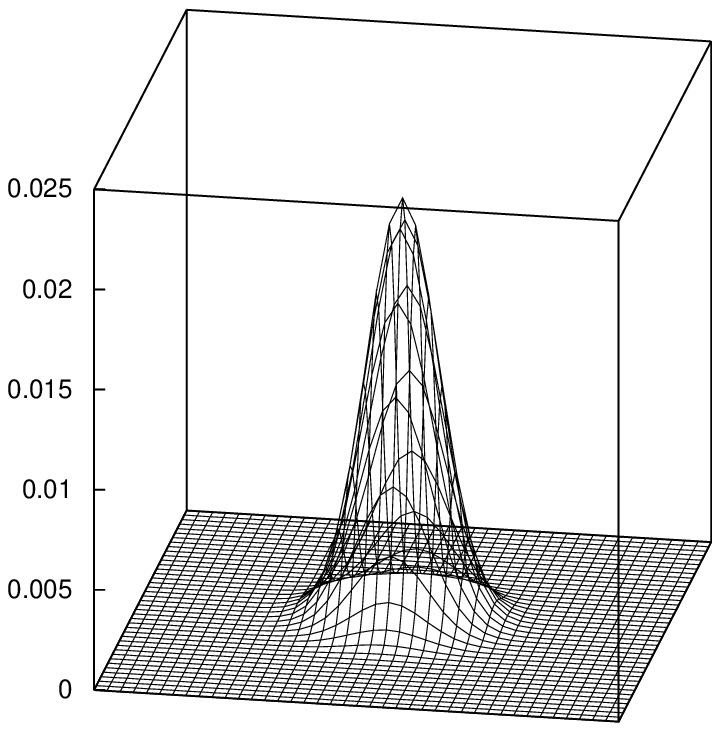}
\hspace{-1.5cm}  
\includegraphics[width=6.5cm,height=4.5cm]{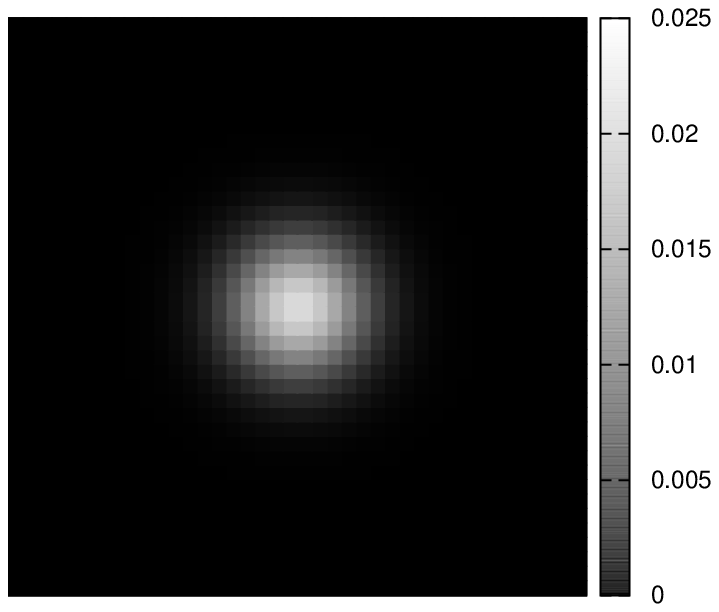}
  
\hspace{1cm}  
  
\vspace{-1.2cm}  
 
\includegraphics[width=6.5cm,height=4.5cm]{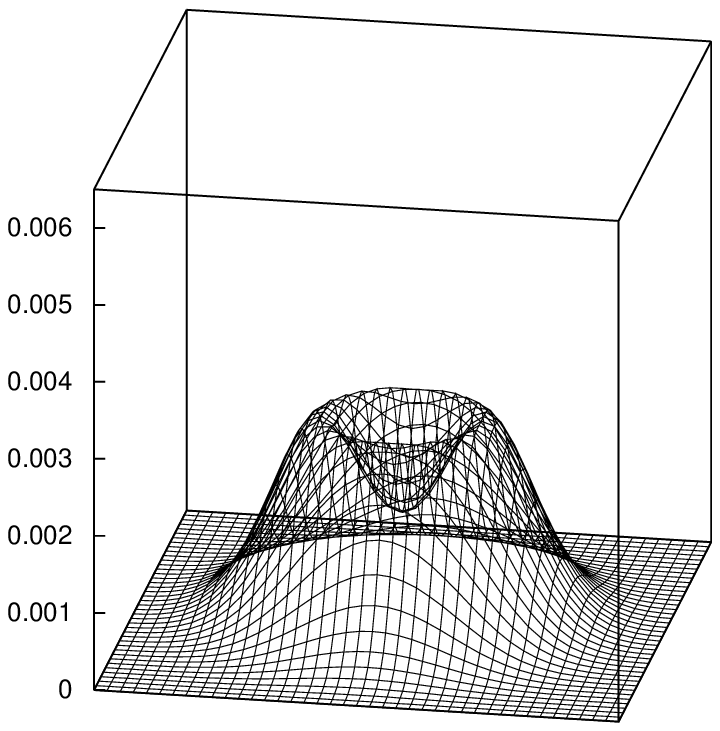}
\hspace{-1.5cm}  
\includegraphics[width=6.5cm,height=4.5cm]{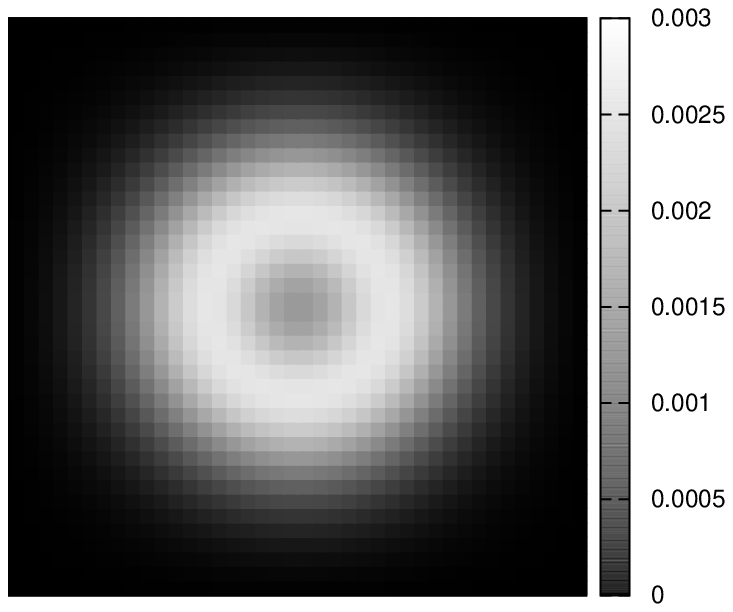}
\caption{The interaction strength is $\kappa=0.5$. }  
\label{fig:seg0.5}  
\end{figure}  
\clearpage
\subsection{$\kappa=10$} 

In the regime of large interaction strength $\kappa$,
 the segregation phenomenon occurs: the supports get more and more disjoint,
see Figure \ref{fig:seg10}.
 
\begin{figure}[h!] 
\includegraphics[width=6.5cm,height=4.5cm]{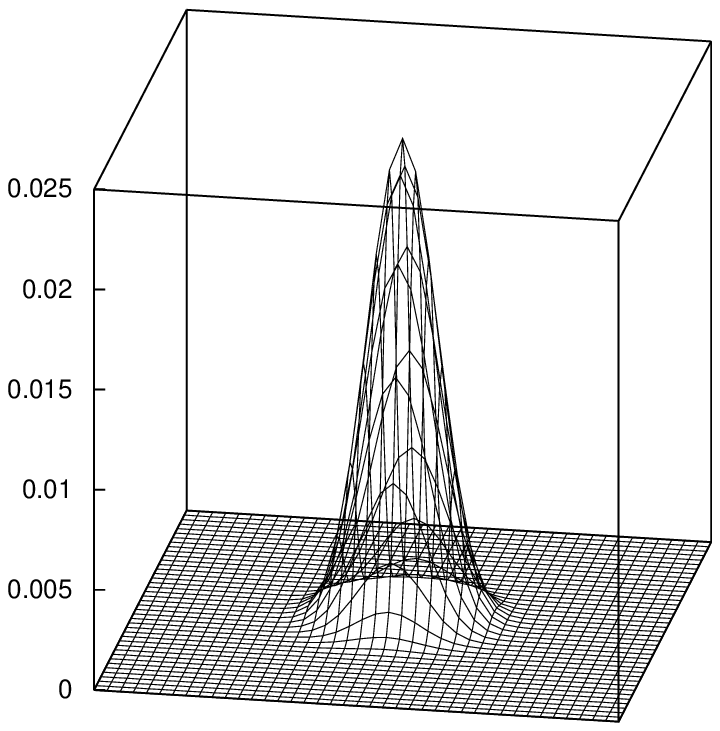}
\hspace{-1.5cm}  
\includegraphics[width=6.5cm,height=4.5cm]{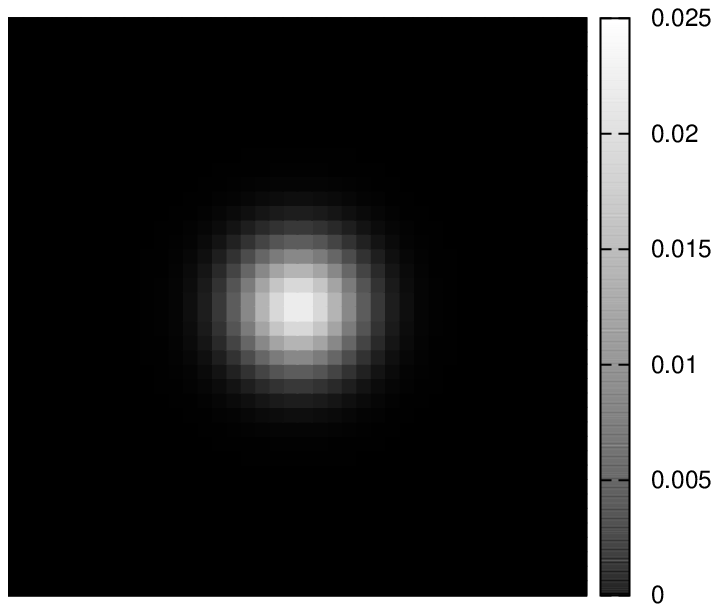}
 
\hspace{1cm} 
 
\vspace{-1.1cm} 
 
\includegraphics[width=6.5cm,height=4.5cm]{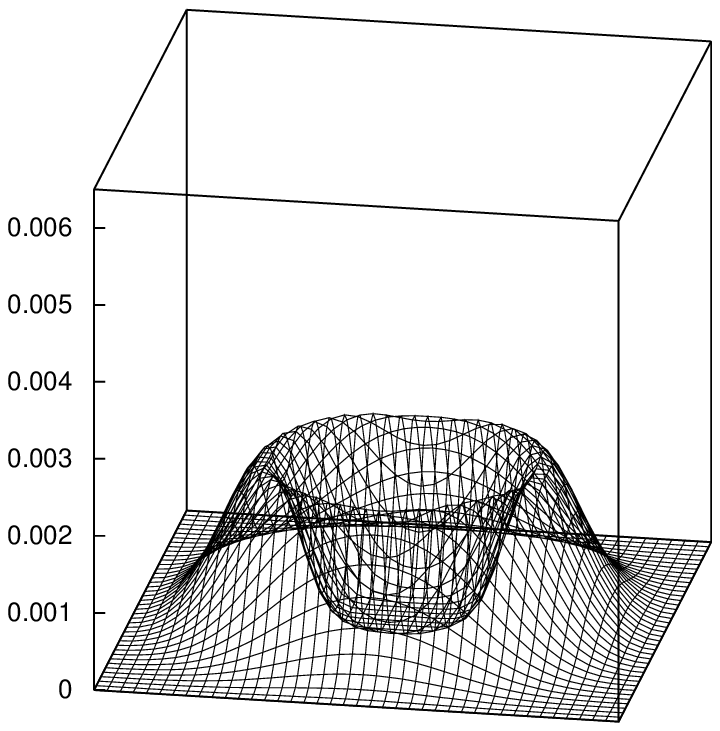}
\hspace{-1.5cm} 
\includegraphics[width=6.5cm,height=4.5cm]{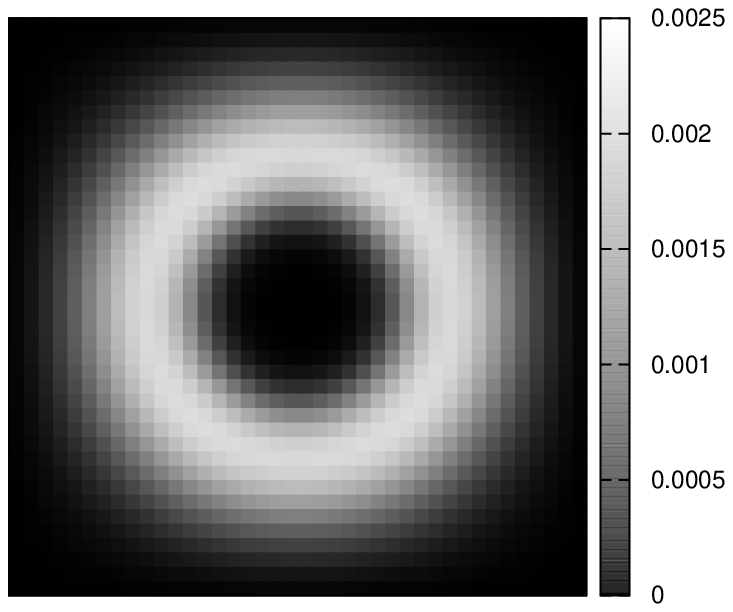}
\caption{The interaction strength is $\kappa=10$.}  
\label{fig:seg10}  
\end{figure}  
 
\brm  
Up to the shape of the support of $\phi_i$, there is no qualitative change in 
the picture if the two harmonic potentials are slightly dislocated with respect 
to each other.
\erm 

\subsection{Coulomb energy}
 
Finally, we monitor the decay of the Coulomb energy 
from formula~\eqref{def:D0}, see Figure \ref{fig:D0}. 
 
\begin{figure}[h!]  
\centering  
\includegraphics[width=6cm,height=5cm]{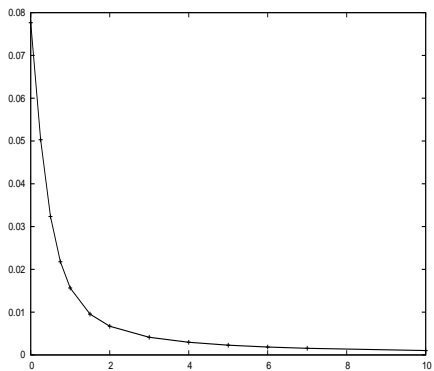}  
\caption{The decay of $\D_0[z_1^\kappa,z_2^\kappa]$ as a function of  
$\kappa$.}  
\label{fig:D0}  
\end{figure}

\clearpage
 
\end{document}